\documentclass{article}

\usepackage[english]{babel}
\usepackage{amsmath,amssymb,amsthm,mathtools}
\usepackage[letterpaper,top=2cm,bottom=2cm,left=3cm,right=3cm,marginparwidth=1.75cm]{geometry}

\usepackage{amsmath}
\usepackage{graphicx}
\usepackage[colorlinks=true, allcolors=blue]{hyperref}
\theoremstyle{plain}
\newtheorem{theorem}{Theorem}[section]
\newtheorem{lemma}[theorem]{Lemma}
\newtheorem{proposition}[theorem]{Proposition}

\theoremstyle{definition}

\newtheorem{question}[theorem]{Question}
\theoremstyle{remark}

\title{\textbf{Convergence Analysis of Greedy Algorithms with Adaptive Relaxation in Hilbert Spaces}}
\author{
  Pablo M. Berná\thanks{Departamento de Matemáticas, CUNEF Universidad, 28040 Madrid (Spain), pablo.berna@cunef.edu}
  \and
  Andrea García\thanks{Universidad San Pablo CEU, CEU Universities, 28003 Madrid (Spain) and Departamento de Matemáticas, CUNEF Universidad, 28040 Madrid (Spain), andrea.garciapons@usp.ceu.es}
}

\date{}
\usepackage{hyperref}
\begin{document}
\maketitle

\begin{abstract}
The Power--Relaxed Greedy Algorithm (PRGA) was introduced as a generalization of the so called Relaxed Greedy Algorithm, introduced by DeVore and Temlyakov, by replacing the relaxation parameter $1/m$ with $1/m^\alpha$, with the aim of improving convergence rates. While the case $\alpha\le 1$ is well understood, the behavior of the algorithm for $\alpha>1$ remained an open problem. In this work, we answer this question and, moreover, we introduce a relaxed greedy algorithm with an optimal step size chosen by exact line search at each iteration. 
\end{abstract}
\section{Introduction}
Greedy algorithms have been widely studied in nonlinear approximation theory, especially in problems involving approximation with respect to redundant systems, commonly referred to as dictionaries (see for instance \cite{BT,DT}). In contrast to orthonormal bases, dictionaries do not require linear independence and therefore provide greater flexibility in representing elements of a Hilbert space. This added flexibility, however, makes the design of efficient approximation methods more challenging, and greedy algorithms offer a natural and intuitive approach in this context.

One of the earliest and most studied procedures is the Pure Greedy Algorithm. This algorithm constructs an approximation iteratively by selecting, at each step, the dictionary element that has the largest correlation with the current residual. The selected atom is then added to the approximation in a straightforward manner. Due to its simplicity, the Pure Greedy Algorithm is easy to implement and computationally attractive. Nevertheless, it is now well understood that, for general dictionaries, this method may lead to slow or non-optimal convergence rates.

To overcome these shortcomings, DeVore and Temlyakov proposed the Relaxed Greedy Algorithm. The main modification introduced by this approach is the use of a relaxation step, in which the new approximation is obtained as a convex combination of the previous approximation and the newly selected dictionary element. This relaxation reduces the aggressiveness of the update and results in a more stable approximation process. A key result associated with this algorithm is that it achieves optimal convergence rates for an important class of functions, namely those with bounded atomic norm.

The introduction of the Relaxed Greedy Algorithm represented an important step in the development of greedy approximation methods. It showed that an appropriate balance between greedy selection and relaxation can restore optimal performance even when working with highly redundant dictionaries. This idea has since played a role in the development of various generalized greedy algorithms and has also highlighted connections with optimization techniques such as the Frank–Wolfe algorithm.

The effectiveness of the Relaxed Greedy Algorithm naturally leads to further questions regarding the choice of the relaxation parameter and the possibility of alternative strategies that may improve performance in specific settings. These considerations have motivated several extensions and variants of the original algorithm and provide the motivation for the present work.

The structure of the paper is as follows. In Section \ref{intro}, we give the main definition fo the diferent algorithms and dicitionaries that we use in the paper. In Section \ref{three}, we present the proof of Theorem \ref{main}, which provides a negative answer regarding the convergence of the PRGA for values of $\alpha > 1$. In Section \ref{four}, we present a numerical simulation of the PRGA implemented in Python for different values of $\alpha > 1$, showing that, indeed, for such values the algorithm does not necessarily converge. Finally, in Section \ref{five}, we introduce the CRGA and establish a convergence result.

\section{Notation and main results}\label{intro}
In this section we introduce the main context of work: dictionaries in Hilbert spaces. Also, we give the exact definition of the Pure Greedy Algorithm and the Relaxed Greedy Algorithm to introduce the main algorithms that we study here.

In this paper, we consider a \textit{Hilbert space}: let $\mathbb H$ be a real or complex vector space endowed with an inner product
\[
\langle \cdot , \cdot \rangle : \mathbb H \times \mathbb H \to \mathbb R \, \text{or}\, \mathbb C,
\]
which satisfies the following properties for all $x,y,z \in \mathbb H$ and all scalars $\alpha \in \mathbb R$ or $\mathbb C$:
\begin{enumerate}
    \item Positivity: $\langle x,x\rangle \ge 0$, with equality if and only if $x=0$.
    \item Conjugate symmetry: $\langle x,y\rangle = \overline{\langle y,x\rangle}$.
    \item Homogeneity in the first argument: $\langle \alpha x, y\rangle = \alpha \langle x,y\rangle$.
    \item Linearity in the first argument: $\langle x+y, z\rangle = \langle x,z\rangle + \langle y,z\rangle$.
\end{enumerate}

The inner product induces a norm on $\mathbb H$ given by
\[
\|x\| = \sqrt{\langle x,x\rangle}.
\]
The space $\mathbb H$ is called a Hilbert space if it is complete with respect to the metric induced by this norm.

A \textit{dictionary} $\mathcal D \subset \mathbb H$ is a collection of elements satisfying the following conditions:
\begin{enumerate}
    \item Each element $g \in \mathcal D$, referred to as an \emph{atom}, has unit norm, i.e., $\|g\|=1$, and whenever $g \in \mathcal D$ one also has $-g \in \mathcal D$.
    \item The closed linear span of $\mathcal D$ is dense in $\mathbb H$. Equivalently, for every $x \in \mathbb H$ and every $\varepsilon > 0$, there exist elements $g_1,\dots,g_n \in \mathcal D$ and scalars $a_1,\dots,a_n$ such that
    \[
    \left\| x - \sum_{i=1}^n a_i g_i \right\| < \varepsilon.
    \]
\end{enumerate}

As a particular example, if $\{h_k\}_{k=1}^\infty$ is an orthonormal basis of $\mathbb H$, then the set $\mathcal D = \{\pm h_k\}_{k=1}^\infty$ constitutes a dictionary in the sense described above.

\subsection{The PGA and RGA}
Let $\mathbb H$ be a real Hilbert space and let $\mathcal D\subset \mathbb H$ be a dictionary. Associated with $\mathcal D$, and for $\tau>0$, we define the approximation class $A_\tau(\mathcal D)$ as the closure in $\mathbb H$ of the set
\[
\left\{ f\in\mathbb H:\ f=\sum_{k\in\Lambda} a_k g_k,\ \ g_k\in\mathcal D,\ |\Lambda|<\infty,\ \text{and}\ \sum_{k\in\Lambda}|a_k|^\tau\le 1\right\}.
\]

In order to introduce the new version of the Relaxed Greedy Algorithm studied in \cite{DT}, we briefly review first the Pure Greedy Algorithm.

Given $f\in\mathbb H$, let $g=g(f)\in\mathcal D$ be an element of the dictionary maximizing $\langle f,g\rangle$ (for simplicity, we assume that such a maximizer exists). We define
\[
G(f)=G(f,\mathcal D):=\langle f,g\rangle\, g,
\qquad
R(f)=R(f,\mathcal D):=f-G(f).
\]

\medskip
\noindent\textbf{Pure Greedy Algorithm (PGA).}
Set $R_0(f)=R_0(f,\mathcal D):=f$ and $G_0(f)=G_0(f,\mathcal D):=0$. For each $m\ge 1$, define recursively
\[
G_m(f)=G_m(f,\mathcal D):=G_{m-1}(f)+G\bigl(R_{m-1}(f)\bigr),
\]
\[
R_m(f)=R_m(f,\mathcal D):=f-G_m(f)=R\bigl(R_{m-1}(f)\bigr).
\]

\medskip
\noindent\textbf{Relaxed Greedy Algorithm (RGA).}
Set $R_0^r(f)=R_0^r(f,\mathcal D):=f$ and $G_0^r(f)=G_0^r(f,\mathcal D):=0$. For $m=1$ we take
\[
G_1^r(f):=G_1(f),\qquad R_1^r(f):=R_1(f).
\]
For $m\ge 2$, we define
\[
G_m^r(f)=G_m^r(f,\mathcal D):=\left(1-\frac{1}{m}\right)G_{m-1}^r(f)+\frac{1}{m}\, g\bigl(R_{m-1}^r(f)\bigr),
\]
\[
R_m^r(f)=R_m^r(f,\mathcal D):=f-G_m^r(f).
\]

Using this algorithm, \cite{DT} proves the following estimate.

\begin{theorem}\label{dtt}
Let $\mathcal D$ be a dictionary in a Hilbert space. Then, for every $f\in A_1(\mathcal D)$ and every $m=1,2,\dots$,
\[
\|f-G_m^r(f)\|\le \frac{2}{\sqrt{m}}.
\]
\end{theorem}

\subsection{Power-Relaxed Greedy Algorithm (PRGA)}
A natural question arising from the definition of the RGA is the following: why is the choice $1/m$ used in the convex combination, rather than $1/m^\alpha$ with $\alpha$ an arbitrary positive real number? Could such a power improve the convergence result proved by DeVore and Temlyakov?

This natural question motivated A. García in \cite{G} to introduce the following generalization of the RGA, which is called the Power-Relaxed Greedy Algorithm.

Set $\mathcal R_0^{r}(f)=\mathcal R_0^{r}(f,\mathcal D):=f$ and $\mathcal T_0^{r}(f)=\mathcal T_0^{r}(f,\mathcal D):=0$. For $m=1$ we define
\[
\mathcal T_1^{r}(f):=G_1(f),\qquad \mathcal R_1^{r}(f):=R_1(f).
\]
For $m\ge 2$, we set
\[
\mathcal T_m^{r}(f)=\mathcal T_m^{r}(f,\mathcal D):=
\left(1-\frac{1}{m^\alpha}\right)\mathcal T_{m-1}^{r}(f)
+\frac{1}{m^\alpha}\, g\bigl(\mathcal R_{m-1}^{r}(f)\bigr),
\]
\[
\mathcal R_m^{r}(f)=\mathcal R_m^{r}(f,\mathcal D):=f-\mathcal T_m^{r}(f).
\]

The convergence of this algorithm is described in the following result, proved in \cite{G}.

\begin{theorem}\label{than}
Let $\mathcal D$ be a dictionary in a Hilbert space. If $\alpha\le 1$, then for every $f\in A_1(\mathcal D)$ and every $m=1,2,\dots$,
\[
\|f-\mathcal T_m^r(f)\|^2\le \frac{4}{m^{\alpha}}.
\]
\end{theorem}
Of course, as this result shows, for $\alpha \leq 1$ the optimal case is $\alpha = 1$, thereby recovering the RGA and the result proved by DeVore and Temlyakov. The question then arises: what happens for $\alpha > 1$? This is precisely the question raised in \cite{G}.

\begin{question}[{\cite[Question 1]{G}}]
Is it possible to find $\alpha_0 > 1$ and $c>0$ such that
\[
\|f - \mathcal T^r_m(f)\| \leq \frac{c}{m^{\alpha_0/2}}, 
\qquad m = 1, 2, \ldots,
\]
for every $f \in A_1(\mathcal D)$?
\end{question}

Here, we answer this question in the negative sense.

\begin{theorem}\label{main}
Taking $\alpha>1$ for the (PRGA), there exists an element $f\in A_1(\mathcal D)$ with $\mathcal D$ a particular dictionary of $\mathbb R^2$ such that 
$$\inf_m\Vert f-\mathcal T_m^r(f)\Vert >0.$$
\end{theorem}

\subsection{The Convex-Relaxed Greedy Algorithm}
As we have seen previously, in the definition of the RGA, we use a convex combination with the factor $1/m$ between $G_{m-1}^r(f)$ and $g(R_{m-1}^r(f))$. The question we ask here is: what if we change the factor $1/m$ and try to study what the best convex combination would be? To address this question, we introduce the following algorithm, which we call the \textbf{Convex-Relaxed Greedy Algorithm} (CRGA): let $\mathcal D$ be a symmetric dictionary in a real Hilbert space $\mathbb H$, and let $f\in\mathbb H$. We define
\[
f_0:=0, \qquad r_0:=f.
\]

\noindent  For each $m\geq 1$, assuming that $f_{m-1}$ and $r_{m-1}$ are defined, we proceed as follows:

\begin{enumerate}
    \item \textbf{Greedy selection.} We choose $g_m\in\mathcal D$ such that
    \[
    \langle r_{m-1}, g_m\rangle
    =\sup_{g\in\mathcal D}\langle r_{m-1}, g\rangle.
    \]

    \item \textbf{Optimal relaxation step.} We define
    \[
    f_m=(1-\gamma_m)f_{m-1}+\gamma_m g_m,
    \]
    where the parameter $\gamma_m\in[0,1]$ is chosen as a solution of the problem
    \[
    \gamma_m \in \arg\min_{\gamma\in[0,1]}
    \left\| f-\bigl((1-\gamma)f_{m-1}+\gamma g_m\bigr)\right\|^2.
    \]

    \item \textbf{Residual.} We define
    \[
    r_m:=f-f_m.
    \]
\end{enumerate}
\textcolor{black}{Observe that, by construction, $f_m \in \operatorname{conv}(\mathcal D)$ for all $m$.} For the CRGA, we have the following result.

\begin{theorem}\label{thcrga}
Let $\mathcal D$ be a dictionary in a real Hilbert space. If
$f\in A_1(\mathcal D)$, then the sequence $(f_m)_{m\ge 0}$ generated by the previous algorithm satisfies
\[
\|f-f_m\| \le \frac{2}{\sqrt{m+4}},
\qquad m=0,1,2,\ldots
\]
\end{theorem}

\section{Proof of Theorem \ref{main}}\label{three}
To prove Theorem \ref{main}, we need to show the following technical lemma.

\begin{lemma}\label{technical}
For $\alpha>1$ we have
\[
P_\alpha:=\prod_{k=2}^\infty\Bigl(1-\frac1{k^\alpha}\Bigr)\in(0,1).
\]
\end{lemma}

\begin{proof}
We write the product in terms of logarithms:
\[
\ln P_\alpha
=\sum_{k=2}^\infty \ln\Bigl(1-\frac1{k^\alpha}\Bigr).
\]
Each factor of the product satisfies $0<1-1/k^\alpha<1$, and therefore
\[
\ln\Bigl(1-\frac1{k^\alpha}\Bigr)<0\qquad (k\ge2),
\]
so that the partial sums
\[
S_N:=\sum_{k=2}^N \ln\Bigl(1-\frac1{k^\alpha}\Bigr)
\]
form a monotone decreasing sequence. In order to conclude that $\ln P_\alpha$
is a real number, it suffices to show that $(S_N)$ is bounded
from below.

To this end, we use a standard bound for the logarithm. For $x\in(0,1)$ one has
\[
\ln(1-x)\ge -\frac{x}{1-x}.
\]
In particular, if $0<x\le \tfrac12$, then $\frac1{1-x}\le2$ and hence
\[
\ln(1-x)\ge -2x.
\]
Taking $x=1/k^\alpha$, 
$1/k^\alpha\le\tfrac12$ for all $k\ge 2$, and therefore
\[
\ln\Bigl(1-\frac1{k^\alpha}\Bigr)
\;\ge\; -\frac{2}{k^\alpha},
\qquad k\ge 2.
\]
Thus, for every $N\ge 2$,
\[
S_N
=  \sum_{k=2}^{N} \ln\Bigl(1-\frac1{k^\alpha}\Bigr)
\;\ge\;
  - 2\sum_{k=2}^{N} \frac1{k^\alpha}.
\]
Since $\alpha>1$, the series $\sum_{k=2}^\infty k^{-\alpha}$ converges, and hence
the partial sums $\sum_{k=2}^{N} k^{-\alpha}$ are uniformly bounded in $N$.
Therefore, there exists a constant $C\in\mathbb R$ such that
\[
S_N \ge C \qquad\text{for all } N\ge 2.
\]
This shows that the sequence $(S_N)$ is decreasing and bounded from below, and
thus converges to some limit $S\in\mathbb R$:
\[
\sum_{k=2}^\infty \ln\Bigl(1-\frac1{k^\alpha}\Bigr) = S.
\]

Finally,
\[
P_\alpha
= \prod_{k=2}^\infty\Bigl(1-\frac1{k^\alpha}\Bigr)
= \exp\Bigl(\sum_{k=2}^\infty \ln\Bigl(1-\frac1{k^\alpha}\Bigr)\Bigr)
= e^S,
\]
and since $S$ is a real number, we have $e^S>0$. Moreover, as all the factors
$1-1/k^\alpha$ are strictly less than $1$, it follows that $P_\alpha<1$.
In conclusion, $P_\alpha\in(0,1)$.
\end{proof}

\begin{proof}[Proof of Theorem \ref{main}]
    Fix $\alpha>1$. Let $\mathbb H=\mathbb R^2$ with the Euclidean inner product and
\[
e_1=(1,0),\qquad e_2=(0,1),\qquad 
\mathcal D=\{\pm e_1,\pm e_2\}.
\]
Choose $b\in(0,1/2)$ and define
\[
f:=(1-b)e_1+be_2.
\]
Then $f\in A_1(\mathcal D)$ and, for the Power--Relaxed Greedy Algorithm with steps $\lambda_m:=m^{-\alpha}$, we will show that 
\[
\|f-\mathcal T_m^r(f)\|_2\ \ge\ \frac{b\,P_\alpha}{\sqrt2}\ >0,
\qquad m\ge 1,
\]
where
\[
P_\alpha:=\prod_{k=2}^\infty\Bigl(1-\frac1{k^\alpha}\Bigr)\in(0,1)
\]
using Lemma \ref{technical}. In particular, $\|f-\mathcal T_m^r(f)\|_2>0$ for every $m$, and the estimate
of Theorem~\ref{than} cannot hold for $\alpha>1$.

Since $f=(1-b)e_1+be_2$ with $(1-b)+b=1$, we have $f\in A_1(\mathcal D)$.
Moreover, because $b<1/2$ we have $\langle f,e_1\rangle=1-b>b=\langle f,e_2\rangle$,
so the first greedy atom is uniquely $g(f)=e_1$ and hence
\[
\mathcal T_1^r(f)=G_1(f)=\langle f,e_1\rangle e_1=(1-b)e_1.
\]

For $m\ge 2$, the algorithm has the form
\[
\mathcal T_m^r(f)=(1-\lambda_m)\mathcal T_{m-1}^r(f)+\lambda_m\,g(\mathcal R_{m-1}^r(f)),
\qquad \lambda_m=\frac1{m^\alpha},
\]
and $g(\mathcal R_{m-1}^r(f))\in\mathcal D$.  Note that every atom in $\mathcal D$
has $\|\cdot\|_1$-norm equal to $1$, i.e. $\|d\|_1=1$ for all $d\in\mathcal D$.

Set $s_m:=\|\mathcal T_m^r(f)\|_1$. Using the triangle inequality for $\|\cdot\|_1$,
\[
s_m=\|(1-\lambda_m)\mathcal T_{m-1}^r(f)+\lambda_m g(\mathcal R_{m-1}^r(f))\|_1
\le (1-\lambda_m)s_{m-1}+\lambda_m\|g(\mathcal R_{m-1}^r(f))\|_1
=(1-\lambda_m)s_{m-1}+\lambda_m.
\]
Define $u_m:=1-s_m$. Then
\[
u_m=1-s_m\ \ge\ 1-\bigl[(1-\lambda_m)s_{m-1}+\lambda_m\bigr]
=(1-\lambda_m)(1-s_{m-1})=(1-\lambda_m)u_{m-1}.
\]
Iterating from $m=2$ gives
\[
u_m\ge u_1\prod_{k=2}^m(1-\lambda_k).
\]
Since $u_1=1-\|\mathcal T_1^r(f)\|_1=1-(1-b)=b$, we get
\[
1-\|\mathcal T_m^r(f)\|_1=u_m\ \ge\ b\prod_{k=2}^m\Bigl(1-\frac1{k^\alpha}\Bigr)
\ \ge\ b\prod_{k=2}^\infty\Bigl(1-\frac1{k^\alpha}\Bigr)=bP_\alpha.
\]

Now consider the unit vector $v=(1,1)/\sqrt2$. For any $x\in\mathbb R^2$,
\[
\langle x,v\rangle=\frac{x_1+x_2}{\sqrt2}\le \frac{|x_1|+|x_2|}{\sqrt2}
=\frac{\|x\|_1}{\sqrt2}.
\]
Hence,
\[
\langle \mathcal T_m^r(f),v\rangle\le \frac{\|\mathcal T_m^r(f)\|_1}{\sqrt2}
\le \frac{1-bP_\alpha}{\sqrt2}.
\]
But $\langle f,v\rangle=\frac{(1-b)+b}{\sqrt2}=\frac1{\sqrt2}$, so
\[
\langle f-\mathcal T_m^r(f),v\rangle
=\langle f,v\rangle-\langle \mathcal T_m^r(f),v\rangle
\ge \frac{bP_\alpha}{\sqrt2}.
\]
Since $\|v\|_2=1$, by Cauchy--Schwarz
\[
\|f-\mathcal T_m^r(f)\|_2\ \ge\ \langle f-\mathcal T_m^r(f),v\rangle
\ \ge\ \frac{bP_\alpha}{\sqrt2}\ >0,
\]
for every $m\ge 1$, as claimed.
\end{proof}

\section{Numerical simulation}\label{four}
To illustrate numerically the lack of convergence of the algorithm,
based on the previous section, we consider $f=(1/2,1/2)$ and perform
simulations of the algorithm for different values of $\alpha$.
The results presented below can be found in \cite{code}.

\begin{verbatim}
--- Results of the simulation ---
Alpha=1.1: Final Error ||f - T_500|| = 0.003805
Alpha=1.5: Final Error ||f - T_500|| = 0.068021
Alpha=2.0: Final Error ||f - T_500|| = 0.177130
\end{verbatim}
\begin{center}
    \includegraphics[scale=0.5]{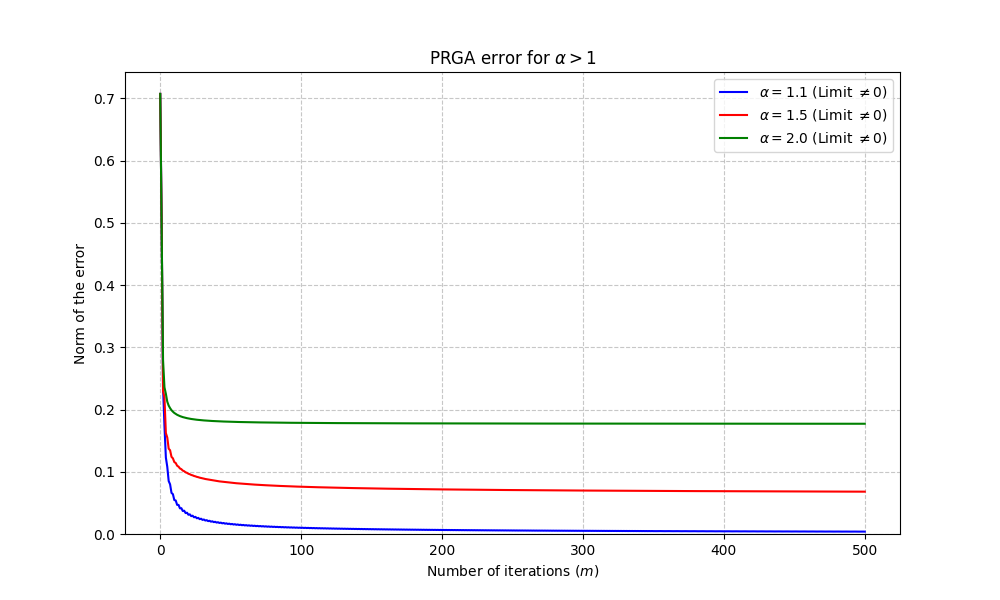}
\end{center}

\section{The Convex-Relaxed Greedy Algorithm}\label{five}
Before studying the convergence result of the CRGA, we will provide an explicit value for the parameter $\gamma$ that defines this algorithm. For this, we need the following lemma.

\begin{lemma}[Optimal choice of the relaxation parameter]\label{lemma:gamma}
Let $r\in\mathbb H$ and $d\in\mathbb H$ with $d\neq 0$. Consider the function
\[
\Phi(\gamma):=\|r-\gamma d\|^2,\qquad \gamma\in\mathbb R.
\]
Then:

\begin{enumerate}
    \item $\Phi$ is strictly convex and its unique global minimizer over $\mathbb R$ is
    \[
    \gamma^\ast=\frac{\langle r,d\rangle}{\|d\|^2}.
    \]

    \item The constrained problem
    \[
    \min_{\gamma\in[0,1]} \Phi(\gamma)
    \]
    has a unique solution $\gamma_m\in[0,1]$, given by
    \[
    \gamma_m=\Pi_{[0,1]}(\gamma^\ast)
    =\min\{1,\max\{0,\gamma^\ast\}\}.
    \]

    \item For every $\gamma\in[0,1]$, the following inequality holds:
    \[
    \Phi(\gamma_m)\le \Phi(\gamma).
    \]
\end{enumerate}
\end{lemma}

\begin{proof}
The function $\Phi$ is a quadratic polynomial in $\gamma$ with leading coefficient
$\|d\|^2>0$, and therefore it is strictly convex. Differentiating, we obtain
\[
\Phi'(\gamma)=-2\langle r,d\rangle+2\gamma\|d\|^2,
\]
and setting the derivative equal to zero yields the unique critical point
\[
\gamma^\ast=\frac{\langle r,d\rangle}{\|d\|^2},
\]
which is the global minimizer of $\Phi$ over $\mathbb R$.

Since $\Phi$ is convex, its minimization over the closed and convex interval
$[0,1]$ is obtained by projecting the global minimizer onto this interval.
This yields the expression
\[
\gamma_m=\Pi_{[0,1]}(\gamma^\ast).
\]

Finally, since $\gamma_m$ is the minimizer of $\Phi$ on $[0,1]$, we have
\[
\Phi(\gamma_m)\le \Phi(\gamma), \qquad \forall\,\gamma\in[0,1],
\]
which completes the proof.
\end{proof}

Then, for the CRGA, we define
\[
d_m:=g_m-f_{m-1}.
\]
For $\gamma\in[0,1]$, we have
\[
f-\bigl((1-\gamma)f_{m-1}+\gamma g_m\bigr)
= r_{m-1}-\gamma d_m.
\]
Therefore, the functional we want to minimize is
\[
\Phi(\gamma)
:=\|r_{m-1}-\gamma d_m\|^2
= \|r_{m-1}\|^2
-2\gamma\langle r_{m-1}, d_m\rangle
+\gamma^2\|d_m\|^2.
\]

This is a strictly convex quadratic function in $\gamma$. Its derivative is
\[
\Phi'(\gamma)
=-2\langle r_{m-1}, d_m\rangle
+2\gamma\|d_m\|^2.
\]
Setting the derivative equal to zero, we obtain the unconstrained minimizer
\[
\gamma^\ast
=\frac{\langle r_{m-1}, d_m\rangle}{\|d_m\|^2},
\]
provided that $d_m\neq 0$. Since the problem is constrained to the interval $[0,1]$, the desired minimizer is the projection of $\gamma^\ast$ onto this interval:
\[
\gamma_m
=\min\left\{ 1, \max\left\{0,
\frac{\langle r_{m-1}, d_m\rangle}{\|d_m\|^2}
\right\}\right\}.
\]
%\Pi_{[0,1]}\!\left(
%\frac{\langle r_{m-1}, d_m\rangle}{\|d_m\|^2}
%\right)
If $d_m=0$, then $g_m=f_{m-1}$ and any $\gamma_m\in[0,1]$ is optimal; in this case we take $\gamma_m=0$.

%Recall that, since $\mathcal D$ is symmetric, we have
%\[
%A_1(\mathcal D)=\overline{\operatorname{conv}}(\mathcal D).
%\]
Also, for the convergence of the CRGA, we need to use the following property.

\begin{lemma}\label{lemma:am-monotone}
Let $\mathcal D$ be a symmetric dictionary in a Hilbert space $\mathbb H$, let
$f\in A_1(\mathcal D)$, and let $(f_m)_{m\ge 0}$ be the sequence generated by the
relaxed greedy algorithm with optimal step size, with $f_0=0$. Define
\[
a_m:=\|f-f_m\|^2,\qquad m\ge 0.
\]
Then the sequence $(a_m)_{m\ge 0}$ satisfies the following properties:
\begin{enumerate}
    \item $(a_m)$ is nonincreasing, i.e.,
    \[
    a_m\le a_{m-1}\qquad \text{for all } m\ge 1.
    \]
    \item $(a_m)$ is uniformly bounded, and in fact
    \[
    0\le a_m\le a_0=\|f\|^2\le 1
    \qquad \text{for all } m\ge 0.
    \]
\end{enumerate}
\end{lemma}

\begin{proof}
By definition of the algorithm, the parameter $\gamma_m\in[0,1]$ is chosen so as to minimize
\[
\|f-\bigl((1-\gamma)f_{m-1}+\gamma g_m\bigr)\|^2
\quad \text{over } \gamma\in[0,1].
\]
Evaluating the above expression at $\gamma=0$ yields
\[
a_m
=\min_{\gamma\in[0,1]}\|f-\bigl((1-\gamma)f_{m-1}+\gamma g_m\bigr)\|^2
\le \|f-f_{m-1}\|^2
= a_{m-1},
\]
which proves the monotonicity.

Since $f\in A_1(\mathcal D)$ and
$\|g\|=1$ for all $g\in\mathcal D$, we have $\|f\|\le 1$. Therefore
\[
a_m\le a_0=\|f\|^2\le 1
\quad \text{for all } m\ge 0,
\]
which completes the proof.
\end{proof}

\begin{proof}[Proof of Theorem \ref{thcrga}]
We define
\[
a_m:=\|r_m\|^2=\|f-f_m\|^2.
\]

\medskip
%=\overline{\operatorname{conv}}(\mathcal D)
\noindent\emph{Step 1.} Since $f\in A_1(\mathcal D)$ and the functional
$g\mapsto\langle r_{m-1},g\rangle$ is linear, we have
\[
\langle r_{m-1}, g_m\rangle
=\sup_{g\in\mathcal D}\langle r_{m-1}, g\rangle
\ge \langle r_{m-1}, f\rangle.
\]
Subtracting $\langle r_{m-1},f_{m-1}\rangle$ from both sides, we obtain
\[
\langle r_{m-1}, g_m-f_{m-1}\rangle
\ge \langle r_{m-1}, f-f_{m-1}\rangle
= \|r_{m-1}\|^2
= a_{m-1}.
\]
Denoting $d_m:=g_m-f_{m-1}$, this yields
\begin{equation}\label{gap}
\langle r_{m-1}, d_m\rangle \ge a_{m-1}.
\end{equation}

\medskip
\noindent\emph{Step 2.} Since $f_{m-1}\in\operatorname{conv}(\mathcal D)$ and $\|g\|=1$ for every
$g\in\mathcal D$, it follows that $\|f_{m-1}\|\le 1$. Consequently,
\begin{equation}\label{dm}
\|d_m\|=\|g_m-f_{m-1}\|\le \|g_m\|+\|f_{m-1}\|\le 2.
\end{equation}

\medskip
\noindent\emph{Step 3.} For $\gamma\in[0,1]$, we define
\[
\Phi(\gamma):=\|r_{m-1}-\gamma d_m\|^2.
\]
By Lemma~\ref{lemma:gamma}, the parameter $\gamma_m$ defined by the algorithm satisfies
\[
a_m=\Phi(\gamma_m)\le \Phi(\gamma),
\qquad \forall\,\gamma\in[0,1].
\]

Using \eqref{gap} and \eqref{dm}, for all $\gamma\in[0,1]$ we obtain
\[
\Phi(\gamma)
= a_{m-1}-2\gamma\langle r_{m-1},d_m\rangle+\gamma^2\|d_m\|^2
\le a_{m-1}-2\gamma a_{m-1}+4\gamma^2.
\]

Using Lemma \ref{lemma:am-monotone}, we know that 
\[
\gamma=\frac{a_{m-1}}{4}\in[0,1].
\]
Then, taking this value, we conclude that
\[
a_m\le a_{m-1}-\frac{a_{m-1}^2}{4}.
\]

\medskip
\noindent\emph{Step 4.} From the previous inequality and using that $\dfrac{1}{1-t}\geq 1+t$ for any $t\in [0,1)$ it follows that
\[
\frac{1}{a_m}
\ge \frac{1}{a_{m-1}}+\frac{1}{4}.
\]
Iterating and using the fact that $a_0=\|f\|^2\le 1$, we obtain
\[
\frac{1}{a_m}\ge \frac{m+4}{4},
\]
which implies
\[
a_m\le \frac{4}{m+4}.
\]
\end{proof}

The optimality of this result is due to the fact that it is known that the bound $m^{-1/2}$ cannot be improved. In fact, it is known that for any greedy-type algorithm one has the following lower bound, which we are going to prove in case the reader is not familiar with it. 

\begin{proposition}\label{prop:lowerbound}
For each $m\in\mathbb N$ with $m\geq 2$ there exist a Hilbert space $\mathbb H$,
a dictionary $\mathcal D\subset\mathbb H$, and an element $f\in A_1(\mathcal D)$ such that
any approximation $s$ using at most $m$ elements from $\mathcal D$ satisfies
\[
\|f-s\|\ \ge\ \dfrac{m^{-1/2}}{\sqrt{2}}.
\]
\end{proposition}

\begin{proof}
Fix $m\in\mathbb N$, $m\ge 2$, and set $n:=2m$. Let $\mathbb H=\mathbb R^n$
equipped with the Euclidean norm, and let
\[
\mathcal D:=\{\pm e_1,\pm e_2,\ldots,\pm e_n\},
\]
where $(e_i)_{i=1}^n$ denotes the canonical basis of $\mathbb R^n$.

Define
\[
f:=\frac{1}{2m}\sum_{i=1}^{2m} e_i.
\]
Then $f$ has a representation of the form
\[
f=\sum_{i=1}^{2m} a_i g_i,
\qquad a_i=\frac{1}{2m},\quad g_i=e_i\in\mathcal D,
\]
and therefore
\[
\sum_{i=1}^{2m} |a_i| = 2m\cdot\frac{1}{2m}=1.
\]
By the definition of the atomic class $A_1(\mathcal D)$, this shows that
$f\in A_1(\mathcal D)$. Moreover,
\[
\|f\|_2^2
= \sum_{i=1}^{2m}\left(\frac{1}{2m}\right)^2
= 2m\cdot\frac{1}{4m^2}
= \frac{1}{2m}\le 1,
\]
so $f$ lies in the unit ball of $\mathbb H$.

Let $s\in\mathbb R^n$ be any vector that can be written as a linear combination of
at most $m$ elements from $\mathcal D$. Equivalently, $s$ has at most $m$ nonzero
coordinates. Denote by $\operatorname{supp}(s)$ the support of $s$, with
$|\operatorname{supp}(s)|\le m$.

For $i\notin \operatorname{supp}(s)$ we have $s_i=0$ and $f_i=1/(2m)$, so
\[
\|f-s\|_2^2
= \sum_{i\notin \operatorname{supp}(s)} |f_i-s_i|^2
 +\sum_{i\in \operatorname{supp}(s)} |f_i-s_i|^2
\;\ge\; \sum_{i\notin \operatorname{supp}(s)} |f_i|^2.
\]
Since at most $m$ indices belong to $\operatorname{supp}(s)$, at least $m$
indices do not, and hence
\[
\|f-s\|_2^2
\;\ge\; m\left(\frac{1}{2m}\right)^2
= \frac{1}{4m}.
\]
Therefore
\[
\|f-s\|_2 \ge \frac{1}{2\sqrt m},
\]
for every $s$ that is a linear combination of at most $m$ elements of $\mathcal D$.

Since the $m$-th iterate produced by any greedy-type algorithm uses at most $m$
elements from $\mathcal D$, the same lower bound applies to the corresponding
approximation.
\end{proof}

\section*{Funding}
Both authors were supported by Grant PID2022-142202NB-I00 (Agencia Estatal de Investigación, Spain).

%\bibliographystyle{siamplain}
%\bibliography{references}
\end{document}